\documentclass[oneside,english]{amsart}
\usepackage[T1]{fontenc}
\usepackage[latin9]{inputenc}
\usepackage{geometry}
\usepackage{amstext}
\usepackage{amsthm}
\usepackage{hyperref}

\hypersetup{
   colorlinks = true,
   urlcolor = blue,
   linkcolor = red
}

\newcommand{\hyp}[5]{\,\mbox{}_{#1}F_{#2}\!\left(
\genfrac{}{}{0pt}{}{#3}{#4};#5\right)}

\newcommand{\Z}{{\mathbb Z}}
\newcommand{\N}{{\mathbb N}}

\newcommand{\C}{{\mathbb C}}

\newcommand{\z}{{\zeta_G}}
\newcommand{\zs}{{\zeta_G^\ast}}

\numberwithin{equation}{section}
\numberwithin{figure}{section}
\newtheorem{thm}{Theorem}

\newtheorem{cor}[thm]{Corollary}
\newtheorem{remark}[thm]{Remark}
\newtheorem{lem}[thm]{Lemma}
\newtheorem{rem}[thm]{Remark}
\newtheorem{prop}[thm]{Proposition}

\numberwithin{equation}{section}

\DeclareMathOperator{\Ai}{Ai}
\newcommand{\pFq}[5]{  {{}_{#1}F_{#2}}\left( \genfrac..{0pt}{}{#3}{#4};#5\right)}

\begin{document}

\title{Multiple zeta values for classical special functions}

\author{Tanay Wakhare$^\ast$, Christophe Vignat$^\dag$}

\address{$^\ast$~University of Maryland, College Park, MD 20742, USA}
\email{twakhare@gmail.com}
\address{$^\dag$~Tulane University, New Orleans, LA 70118, USA and Universit\'{e} Paris Sud, France}
\email{cvignat@tulane.edu, christophe.vignat@u-psud.fr}






\begin{abstract}We compute multiple zeta values (MZVs)
built from the zeros of various entire functions, usually special functions with physical relevance. 
In the usual case, MZVs and their linear combinations are
evaluated using a morphism between symmetric functions and multiple zeta values. 
We show that this technique can be extended to the zeros of any entire function, and 
as an illustration, we explicitly compute some MZVs based on the
zeros of Bessel, Airy, and Kummer hypergeometric functions. We highlight several approaches to the theory of MZVs, such as exploiting the orthogonality of various polynomials and 
fully utilizing the Weierstrass representation of an entire function. On the way, an identity for 
Bernoulli numbers by Gessel and Viennot is 
revisited and generalized to Bessel-Bernoulli polynomials, 
and the classical Euler identity between the Bernoulli numbers and Riemann zeta function at even argument is extended to this same class.
\end{abstract}
\maketitle

\def \sha{{\,\amalg\hskip -3.6pt\amalg\,}}

\section{Introduction}
\label{section1}
Since their introduction in the $1990s$, multiple zeta values have come to be seen as fundamental to many disparate topics \cite{ZagierSurvey}, such as quantum field theory \cite{Broadhurst} or knot invariants \cite{Dunin}. A natural generalization is \textbf{generalized multiple zeta functions}, which are built from the zeros of an arbitrary function. Most of the previous work on multiple zeta values has centered on the prototypical Riemann multiple zeta function. However, over the past decade, work has appeared focusing on zeta functions built from other special functions. Zeta functions built from the roots of Kummer, Airy, and Bessel functions have appeared in scattered works ranging from statistics \cite{Flajolet} to quantum mechanics \cite{Crandall}. We extend the study of these zeta functions into the realm of multiple zeta values, and compute many special cases.

We begin with some basic definitions. One of the most important functions in mathematics is the Riemann zeta function $\zeta(s):=\sum_{n=1}^\infty \frac{1}{n^s}$. We can reinterpret this expression as a formal rescaled sum over the zeros of the trigonometric sine function, and consider more general zeta functions built from the zeros of classical special functions. Given some function $G(z)$, we denote by $\{a_n\}$ the set of zeros (assumed to be non-zero complex numbers) of $G$ indexed by the natural numbers. Then we define a \textbf{zeta function} associated with $G$ as
\begin{equation}
\z(s) = \sum_{n=1}^{\infty} \frac{1}{a_n^s}.
\end{equation}
We can also construct a \textbf{multiple zeta function} as 
\begin{equation}
\label{MVZ}
\z(s_1,\ldots,s_r) = \sum_{n_1 > n_2 > \cdots > n_r \geq 1} \frac{1}{a_{n_1}^{s_1}\cdots a_{n_r}^{s_r}}
\end{equation}
and a \textbf{multiple zeta-star function} as 
\begin{equation}
\zs(s_1,\ldots,s_r) = \sum_{n_1 \geq n_2 \geq \cdots \geq n_r \geq 1} \frac{1}{a_{n_1}^{s_1}\cdots a_{n_r}^{s_r}}.
\end{equation}
The \textbf{weight} of a multiple zeta function is $\sum_{i=1}^r s_i$ and the \textbf{depth} is $r$. For convenience, we let $\z(\{s\}^n)$ denote $\z(\{s,s,\ldots,s\})$ where $s$ is repeated $n$ times. We let the same convention hold for $\zs$.
When we take $G(z)=\frac{1}{\Gamma(1-z)}$, which has zeros at $\{1,2,3,\ldots\}$, we recover a scaled version of the familiar Riemann zeta function. 
We use $\zeta(s_1,\ldots,s_r)$ and $\zeta^\ast(s_1,\ldots,s_r)$, without the subscript $G$, to denote the classical Riemann multiple zeta functions (the case $a_n=n$). Multiple zeta values (MZVs) and multiple zeta star values (MZSVs) refer to evaluations of these functions at positive integral arguments.


Given these definitions, we can define a "quantum zeta function" attached to any quantum system which gives us tight bounds on the ground state energy of the system. Formally, given a system with eigenvalues $\{E_0,E_1,\ldots\}$, arranged in decreasing order of magnitude,  the quantum zeta function $Z(s)$ is defined as the series
\begin{equation}
Z(s) = \sum_{n=0}^\infty \frac{1}{E_n^s}.
\end{equation}
For $s\in \C$ with sufficiently large real part, this sum will converge absolutely in most cases. Importantly, in some situations where no energy eigenvalues can be exactly calculated in terms of elementary special functions, the quantum zeta function still may have a convenient closed form. Given this closed form, the ground state energy $E_0$ can be tightly approximated as $Z(s)^{\frac{1}{s}}$ for large $s$, since the term $\frac{1}{E_0^s}$ will dominate in the sum $Z(s)$. This was exploited in the case of the "quantum bouncer" in \cite{Crandall}: this system is the solution of the Schr\"{o}dinger equation associated with the potential
\[
V\left( x \right) = 
\begin{cases}
+\infty,& \,\, x<0,\\
kx,& \,\, x\ge 0, k>0.
\end{cases}
\]
Since in this case the energy eigenvalues are proportional to the zeros of the Airy function,
$Z(s)$ is a rescaled version of the zeta function associated to the Airy function. 

Now, the knowledge of a quantum multiple zeta function $Z(s_1,s_2,\ldots,s_r),$ as defined in the obvious way, allows us to extract, using the same  asymptotic mechanism, an approximated value of the product $E_1E_2\cdots E_r$ since
\[
Z(\{s\}^r)^\frac{1}{s} \sim \frac{1}{E_1E_2\cdots E_r}  
\]
for large $s.$ This means that even when  energy eigenvalues  cannot be calculated exactly,  close approximations can be extracted for their partial products by considering the quantum multiple zeta function. Thus, information about the zeta and multiple zeta values associated to special functions will inform future work on many quantum systems, due to the fundamental connection between many simple systems and classical special functions. To our surprise,  this result seems to be new.

Additionally, MZVs arise in the calculation of renormalization constants for Feynman diagrams, see \cite{Broadhurst}. While previous work with physical applications has centered on the Riemann case, we consider more exotic MZVs which are built from the zeros of other classical special functions. This will inform future work on quantum systems which depend on the zeros of various special functions. While doing this, we address two important questions in this paper:
\begin{itemize}
\item given a function and its zeros, what does the knowledge of the Weierstrass product factorization of this function tell us about the MZV built on these zeros ?
\item what is the benefit of introducing generalized Bernoulli numbers in the  computation of the MZV's?
\end{itemize}


In Section \ref{section2}, we introduce our various special functions, and discuss the basic product/sum duality that yields multiple zeta values. We also discuss an extension of previous results to entire functions of arbitrary genus, which have not previously been considered in the literature. In Section \ref{section3}, we study the hypergeometric zeta function, and explicitly calculate the first few multiple hypergeometric zeta values. Section \ref{section4} forms the bulk of this paper; we study the Bessel zeta function, which reduces to the Riemann zeta function as a special case. We derive Bessel analogs of the closed form of $\zeta(2n)$ and the Gessel-Viennot Bernoulli number identity. Finally, in Section \ref{section5} we consider the Airy zeta function, and introduce Airy Bernoulli numbers to facilitate the computation of the associated MZVs.

\section{Definitions and Methodology}\label{section2}
Our methodology is based on a  high level sum/product duality which has been explored since the time of Euler. We give a short synopsis here, along with all the definition we will need throughout this paper.

Given a set of formal indeterminates $\{x_i\}$, we can define symmetric functions which are unchanged by any permutation of these variables. Following Macdonald \cite[(Chapter 1)]{Macdonald}, three fundamental symmetric functions are defined, each set of which is a basis for the ring of symmetric polynomials. The elementary symmetric functions are defined as $e_r = \sum_{i_1 < i_2 < \ldots < i_r }x_{i_1}x_{i_2}\cdots x_{i_r},$ the complete symmetric functions are defined as $h_r = \sum_{i_1 \leq i_2 \leq \ldots \leq i_r }x_{i_1}x_{i_2}\cdots x_{i_r}$, and the power sums are defined as $p_r = \sum_i x_i^r$. We also require $e_0=h_0=p_0=1$. The key insight of many previous authors is the striking resemblance between the symmetric functions and multiple zeta (star) values. In fact, under the correct specialization of the $\{x_i\}$,  MZVs can be exactly recovered, as encoded by the following obvious lemma.

\begin{lem}\label{lemma1}
When $x_i = \frac{1}{a_i^s}$, $p_k=\z(ks)$, $e_k = \z(\{s\}^k)$, and $h_k =  \zs(\{s\}^k)$.
\end{lem}

The symmetric functions satisfy the following fundamental relations:
\begin{align}
E(t) &= \sum_{k=0}^{\infty}e_k t^k = \prod_{i=1}^{\infty}\left(1+tx_i\right), \label{eprod} \\
H(t) &= \sum_{k=0}^{\infty}h_k t^k = \prod_{i=1}^{\infty}\left(1-tx_i\right)^{-1} = E(-t)^{-1}, \label{hprod} \\
P(t) &= \sum_{k=1}^{\infty}p_k t^{k-1} = \sum_{i=1}^{\infty}\frac{x_i}{1-tx_i} = \frac{H'(t)}{H(t)} = \frac{E'(-t)}{E(-t)} \label{pprod}.
\end{align}
Because of the sum/product duality inherent to symmetric functions,  various product identities can be used to recover values for $e_k$ and $h_k$, which correspond to special values of multiple zeta functions. While the connection between MZVs and symmetric polynomials has been systematically exploited to prove various general theorems, such as the sum and duality theorems \cite{Zudilin, Ihara, Henderson}, we are concerned here with several special cases. If enough information is available about the series expansion of a generating product $$ \prod_{i=1}^{\infty}\left(1+\frac{t}{a_i^s}\right),$$
for instance under the form of a Weierstrass factorization,  various multiple zeta values can be extracted. This approach has been implicitly pursued by many authors, but very infrequently stated explicitly.

The main innovation of this paper is to extend this basic sum/product duality to entire functions of higher order, which have Weierstrass factorizations with higher order exponential terms. While the theory of infinite products and MZVs have been inextricably linked since the time of Euler, we believe the Weierstrass factorization has been underutilized. 
More formally, let $\psi(z)$ be an entire function with zeros $\{z_1,z_2,\ldots\}$ indexed by decreasing order of magnitude. Then if $\psi$ has 
rank $p$ (i.e. $p$ is the smallest integer $\geq 0$ such that $\zeta_\psi(p+1)$ converges), we have 
Weierstrass canonical representation \cite{Simon}

\begin{equation}\label{Weierstrass}
\psi(z) = z^m e^{Q(z)} \prod_{n=1}^\infty \left( 1- \frac{z}{z_n} \right) \exp\left( \frac{z}{z_n} + \cdots + \frac{1}{p}\left(\frac{z}{z_n} \right)^p \right),
\end{equation}
where $Q(z)$ is some polynomial of degree $\rho \le p.$

Start by taking a logarithmic derivative:
\begin{align}
\frac{\psi'(z)}{\psi(z)} = \frac{m}{z} + Q'(z) + \sum_{n=1}^\infty \frac{-1}{z_n-z} + \frac{1}{z_n}\left(1+ \frac{z}{z_n} + \cdots + \left(\frac{z}{z_n} \right)^{p-1} \right).
\end{align}
Expanding the geometric series in some sufficiently small neighborhood of $0$, where $$\frac{1}{z_n-z} = \frac{1}{z_n}  \left(1+ \frac{z}{z_n} + \cdots \right),$$
the summation reduces to 
\begin{equation}
 \sum_{n=1}^\infty \frac{-1}{z_n-z} + \frac{1}{z_n}\left(1+ \frac{z}{z_n} + \cdots + \left(\frac{z}{z_n} \right)^{p-1} \right) = -  \sum_{n=1}^\infty \sum_{k= p}^\infty \frac{z^k}{z_n^{k+1}} =- \sum_{k=p}^\infty z^k \zeta_\psi(k+1).
\end{equation}
We then have the zeta generating function 
\begin{equation}
\sum_{k=p}^\infty z^k \zeta_\psi(k+1) = \frac{m}{z} + Q'(z) - \frac{\psi'(z)}{\psi(z)}.
\end{equation}

This is essentially the duality between elementary symmetric functions (with generating function $E(t)$) and power sums (with generating function $P(t)$), which is classically $P(t) = \frac{E'(-t)}{E(-t)}$ \eqref{pprod}. However, here we have the correction factor $\frac{m}{z} + Q'(x)$ which does not appear in the prior case. From the 
Weierstrass factorization \eqref{Weierstrass},  various MZVs can also be calculated. Start by normalizing 
\begin{equation}
\Psi(z):=\frac{\psi(z)}{z^m e^{Q(z)}} =  \prod_{n=1}^\infty \left( 1- \frac{z}{z_n} \right) \exp\left( \frac{z}{z_n} + \cdots + \frac{1}{p}\left(\frac{z}{z_n} \right)^p \right),
\end{equation}
which is a product solely over (non zero) zeros. Hence the Weierstrass factorization indicates precisely which correction factor we need to normalize an entire function by in order to extract its associated MZVs. In particular, this is the reason why we consider normalized Bessel functions of the first kind \eqref{eq:j_nu} in order to compute Bessel MZVs. Without the Weierstrass factorization, this would be extremely difficult to derive. Additionally, given such a normalized product,  MZVs can be extracted using a convenient property of roots of unity: for $\omega$ a primitive $(p+1)$-th root of unity, $\sum_{k=0}^{p} \omega^{kl} = 0$ for $1\leq l \leq p$, which can be verified using a geometric series. Therefore, for sufficiently small $z$ we calculate 
\begin{align}
\prod_{k=0}^p \Psi(\omega^k z) &= \prod_{n=1}^\infty   \prod_{k=0 }^p\left({1 - \frac{z \omega^k }{z_n}} \right)\times \exp\left( \sum_{k=0}^p\frac{z\omega^k}{z_n} + \cdots + \frac{1}{p}\left(\frac{z\omega^k}{z_n} \right)^p \right) \\
&=  \prod_{n=1}^\infty \left( 1 -  \frac{(-1)^{p+1}z^{p+1}}{z_n^{p+1}} \right)\\
&= \sum_{k\ge0}\left( -1 \right)^{kp}z^{k\left( p+1 \right)}\zeta_\psi(\{p+1\}^k).\label{psigenfunc}
\end{align}
We take a $(p+1)$-fold product to eliminate the exponential term, and the last equality follows from Equation \eqref{eprod}.
Taking the reciprocal (with a minus sign) then gives us a generating product for $\zeta_\psi^\ast(\{p+1\}^n)$. This intuitively explains some well-known properties of MZVs -- for instance, in the classical Riemann zeta case  we have compact closed form expressions for $\zeta(\{2\}^n)$ but not $\zeta(\{3\}^n)$. Therefore, we expect entire functions of order $p$ to have easily obtainable MZVs at multiples of $p+1$ but not at other arguments, since there usually is no obvious factorization of the generating product for other repeated arguments. The Weierstrass factorization thus has two large advantages which have not been stressed in the literature:
\begin{itemize}
\item it gives us the correct prefactor $z^m e^{Q(z)}$ to normalize our original function by,
\item it allows us to compute MZVs and MZSVs as a $(p+1)$-fold product.
\end{itemize}

We stress that the special case where $p=1$ frequently appears. Then, the computation of $\zeta_G(\{2\}^n)$ is equivalent to the computation of a series expansion for $\Psi(z)\Psi(-z)$, which is often obtainable using known integral representations or direct summation arguments. For example, consider the rank one function $\psi(z) = \frac{1}{\Gamma(1-z)}$, which gives the Riemann zeta function. We obtain the normalized product $\Psi(z)\Psi(-z) = \frac{\sin(\pi z)}{\pi z}$, from which it follows that $\zeta(\{2\}^n) = \frac{\pi^{2n}}{(2n+1)!}$.

Let us now give a brief summary of the special functions we consider. We heavily depend on the theory of hypergeometric functions. Define the \textbf{Pochhammer symbol} as $(z)_n := (z)(z+1)\cdots(z+n-1),\ (z)_0:=1,\ z\in\C$ and the associated shorthand $(a_1,a_2,\ldots,a_r)_n:=(a_1)_n(a_2)_n\cdots (a_r)_n$. Then define the \textbf{hypergeometric function} ${}_rF_s: \C^r\times(\C\setminus-\N_0)^s\times(\C\setminus(1,\infty))\to\C$ as  the series \cite[(9.14, Page 1010)]{Table}
\[
\pFq{r}{s}{a_1 , \cdots , a_r}{b_1 , \cdots , b_s}{z}
:= \sum_{p=0}^{\infty}\frac{{(a_{1},\ldots,a_{r})_{p}}}
{{(b_{1},\ldots,b_{s})_{p}}}\frac{z^{p}}{p!},
\]
for $|z|<1$ and elsewhere by analytic continuation.

We next consider the Bessel function of the first kind, which can be defined in terms of the series expansion \cite[(8.401, Page 910)]{Table} 
\[
J_\nu(z) = \frac{z^\nu}{2^\nu}\sum_{k=0}^\infty  \frac{(-1)^k}{k!\Gamma(\nu+k+1)}\left(\frac{z}{2}\right)^{2k}.
\]
This leads to the definition of the modified Bessel function $I_\nu(x) := i^{-\nu}J_\nu(ix).$ Finally, we consider the Airy function defined in terms of the modified Bessel function \cite[xxxviii]{Table} by 
\[
\Ai\left(z\right)=\frac{\sqrt{z}}{3}\left(I_{-\frac{1}{3}}\left(\xi\right)-I_{\frac{1}{3}}\left(\xi\right)\right),
\]
with $\xi=\frac{2}{3}z^{\frac{3}{2}}$.

\section{Hypergeometric zeta}\label{section3}

As our first special case, we study the \textbf{hypergeometric zeta function}, which is defined by
\[
\zeta_{a,b}\left(s\right)=\sum_{k\ge1}\frac{1}{z_{a,b;k}^{s}},
\]
where $z_{a,b;k}$ are the zeros of the classic Kummer function
\[
\Phi_{a,b}\left(z\right):={}_{1}F_{1}\left(\begin{array}{c}
a\\
a+b
\end{array};z\right).
\]
This function is real-valued so that the zeros $z_{a,b;k}$ are pair-wise complex conjugates.
There are several important special cases; for $a=1,b=0$, we have
\[
\Phi_{1,0}\left(z\right) = _{1}F_{1}\left(\begin{array}{c}
1\\
1
\end{array};z\right)=e^{z}.
\]
Under the specialization $b=a$, we have the result 
\[
\Phi_{a,a}\left(z\right) ={}_{1}F_{1}\left(\begin{array}{c}
a\\
2a
\end{array};z\right)=e^{\frac{z}{2}}2^{a-\frac{1}{2}}\Gamma\left(a+\frac{1}{2}\right)\frac{I_{a-\frac{1}{2}}\left(\frac{z}{2}\right)}{\left(\frac{z}{2}\right)^{a-\frac{1}{2}}}.
\]
Up to a constant factor this has the same zeros as the Bessel $J$ function, which can be seen by comparing series expansions. However, note the presence of the exponential $e^{\frac{z}{2}}$, which means that we cannot derive results about Bessel MZVs as a straightforward limit of the Kummer case. This function has the Weierstrass factorization  \cite[Thm 2.6]{Christophe2}
\begin{equation}
_{1}F_{1}\left(\begin{array}{c}
a\\
a+b
\end{array};z\right)=e^{\frac{a}{a+b}z}\prod_{k\ge1}\left(1-\frac{z}{z_{a,b;k}}\right)e^{\frac{z}{z_{a,b;k}}}.\label{eq:factorization}
\end{equation}
We deduce a (nontrivial) generating function for $\zeta_{a,b}\left(s\right)$ as \cite[Prop 3.1]{Christophe2}
\[
\sum_{k=1}^{\infty}\zeta_{a,b}\left(k+1\right)z^{k}=\frac{b}{a+b}\left(\frac{\Phi_{a,b+1}\left(z\right)}{\Phi_{a,b}\left(z\right)}-1\right).
\]
Moreover, from \eqref{eq:factorization} we deduce
\begin{equation}
_{1}F_{1}\left(\begin{array}{c}
a\\
a+b
\end{array};i z\right){}_{1}F_{1}\left(\begin{array}{c}
a\\
a+b
\end{array};-i z\right)=\prod_{k\ge1}\left(1+\frac{z^{2}}{z_{a,b;k}^{2}}\right).\label{eq:infinite product}
\end{equation}
This expansion can be used to compute the values $\zeta_{a,b}(\left\{2\right\}^n)$
using an identity originally due to Ramanujan \cite[Entry 18, p.61]{Ramanujan} and later proved by Preece \cite{Preece}.

\begin{lem}
\label{prop:Prop1}  We have the hypergeometric identity
\begin{align}
_{1}F_{1}\left(\begin{array}{c}
a\\
a+b
\end{array};i z\right){}_{1}F_{1}\left(\begin{array}{c}
a\\
a+b
\end{array};-i z\right) & =\thinspace_{2}F_{3}\left(\begin{array}{c}
a,b\\
a+b,\frac{a+b}{2},\frac{a+b+1}{2}
\end{array};-\frac{z^{2}}{4}\right).\label{eq:ramanujan}
\end{align}
\end{lem}

Beginning with this expansion, we can iteratively compute even multiple zeta values.
\begin{prop}
The multiple zeta value $\zeta_{a,b}\left(\left\{ 2\right\} ^{n}\right)$ is equal to
\[
\zeta_{a,b}\left(\left\{ 2\right\} ^{n}\right)=\frac{\left(-1\right)^{n}}{n!}\frac{\left(a\right)_{n}\left(b\right)_{n}}{\left(a+b\right)_{n}\left(a+b\right)_{2n}}.
\]
For example,
\[
\zeta_{a,b}\left(\left\{ 2\right\} \right)=-\frac{ab}{\left(a+b\right)^{2}\left(a+b+1\right)}
\]
and
\[
\zeta_{a,b}\left(\left\{ 2,2\right\} \right)=\frac{a(1+a)b(1+b)}{2(a+b)^{2}(1+a+b)^{2}(2+a+b)(3+a+b)}.
\]
\end{prop}
\begin{proof}
We use the result of Lemma \ref{prop:Prop1} and identity \eqref{eq:infinite product}
along with the generating function (see \eqref{psigenfunc})
\[
\prod_{k\ge1}\left(1+\frac{z^{2}}{z_{a,b;k}^{2}}\right)=\sum_{n\ge0}\zeta_{a,b}\left(\left\{ 2\right\} ^{n}\right)z^{2n}.
\]
\end{proof}

The values of  $\zeta_{a,b}(\{2r\}^n)$ can be recursively computed from  $\zeta_{a,b}(\{2\}^n)$ for $r= 1,2,\ldots$ using the following general result.
\begin{thm}\label{dissectionthm}
\label{recursive1}
Take $m\in \Z$, $m> 0$ and $\omega$ a primitive $m-$th root of unity. Then
\begin{equation}
\z(\{ms\}^n) = (-1)^{n(m-1)}  \sum_{l_0+l_1+\cdots+ l_{m-1}=mn} \prod_{j=0}^{m-1} \z(\{s\}^{l_j}) \omega^{jl_j},
\end{equation}
where the $l_i$ are nonnegative integers.
\begin{proof}
We begin with the identity $1-z^m = \prod_{j=0}^{m-1} (1-z\omega^j)$ for $m>0$. Then, taking a scaled generating function for MZVs yields
\begin{align*}
\sum_{n=0}^\infty (-1)^n \z(\{ms\}^n) t^{mn} &= \prod_{k=1}^\infty \left(1 - \frac{t^m}{a^{ms}_k}\right) 
= \prod_{k=1}^\infty \prod_{j=0}^{m-1}  \left(1 - \frac{t\omega^j}{a^s_k }\right) \\
&=  \prod_{j=0}^{m-1}  \prod_{k=1}^\infty  \left(1 - \frac{t\omega^j}{a^s_k }\right) 
= \prod_{j=0}^{m-1 }\left(\sum_{n=0}^\infty \omega^{jn}(-1)^n \z(\{s\}^n) t^{n} \right).
\end{align*}
Comparing coefficients of $t^{mn}$ and simplifying completes the proof.
\end{proof}
\end{thm}


Though this analysis appears intractable for $r\ge 3$, for $r=2$ we have an attractive closed form in terms of a terminating hypergeometric ${}_6F_5$ function evaluated at $-1$, which may be summable for some particular values of $a$ and $b$. 
\begin{prop}\label{hyp4}
The multiple zeta value $\zeta_{a,b}\left(\left\{ 4\right\} ^{n}\right)$
is equal to
\begin{align*}
\zeta_{a,b}\left(\left\{ 4\right\} ^{n}\right) & =\frac{\left(-1\right)^{n}}{\left(2n\right)!}\frac{\left(a\right)_{2n}\left(b\right)_{2n}}{\left(a+b\right)_{2n}\left({a+b}\right)_{4n}}\\
 & \times {}_{6}F_{5}\left(\begin{array}{c}
-2n,1-2n-a-b,1-2n-\frac{a+b}{2},1-2n-\frac{a+b+1}{2},a,b\\
1-2n-a,1-2n-b,a+b,\frac{a+b}{2},\frac{a+b+1}{2}
\end{array};-1\right).
\end{align*}
For example,
\[
\zeta_{a,b}\left(\left\{ 4\right\} \right)=-\frac{ab(a^{3}+a^{2}(1-2b)+b^{2}(1+b)-2ab(2+b)}{(a+b)^{4}(1+a+b)^{2}(2+a+b)(3+a+b)}.
\]
\end{prop}
\begin{proof}
We want to compute $\zeta_{a,b}\left(\left\{ 4\right\} ^{n}\right)$ from the generating function
\begin{align*}
\sum_{n\ge0}\zeta_{a,b}\left(\left\{ 4\right\} ^{n}\right)z^{4n} & =\prod_{k\ge1}\left(1+\frac{z^{4}}{z_{a,b;k}^{4}}\right)=\prod_{k\ge1}\left(1+\frac{\left(\sqrt{i}z\right)^{2}}{z_{a,b;k}^{2}}\right)\prod_{k\ge1}\left(1+\frac{\left(i\sqrt{i}z\right)^{2}}{z_{a,b;k}^{2}}\right).
\end{align*}
In fact, this is the $m=s=2$ case of Theorem \ref{dissectionthm} with $t\mapsto t^2$. Since
\begin{align*}
\prod_{k\ge1}\left(1+\frac{z^{2}}{z_{a,b;k}^{2}}\right) & =\thinspace_{2}F_{3}\left(\begin{array}{c}
a,b\\
a+b,\frac{a+b}{2},\frac{a+b+1}{2}
\end{array};-\frac{z^{2}}{4}\right),
\end{align*}
we deduce
\begin{align*}
\sum_{n\ge0}\zeta_{a,b}\left(\left\{ 4\right\} ^{n}\right)z^{4n}& =\thinspace_{2}F_{3}\left(\begin{array}{c}
a,b\\
a+b,\frac{a+b}{2},\frac{a+b+1}{2}
\end{array};-i\frac{z^{2}}{4}\right)
\thinspace_{2}F_{3}\left(\begin{array}{c}
a,b\\
a+b,\frac{a+b}{2},\frac{a+b+1}{2}
\end{array};i\frac{z^{2}}{4}\right).
\end{align*}

It is a general result for the product of two $_2F_3$ hypergeometric functions that
\begin{align}
_{2}F_{3}\left(\begin{array}{c}
a_{1},a_{2}\\
b_{1},b_{2},b_{3}
\end{array};cz\right)&{}_{2}F_{3}\left(\begin{array}{c}
\alpha_{1},\alpha_{2}\\
\beta_{1},\beta_{2},\beta_{3}
\end{array};dz\right) =\sum_{k\ge0}\frac{z^{k}}{k!}d^{k}\frac{\left(\alpha_{1}\right)_{k}\left(\alpha_{2}\right)_{k}}{\left(\beta_{1}\right)_{k}\left(\beta_{2}\right)_{k}\left(\beta_{3}\right)_{k}}\label{eq:6F5}\\
 &\times  {}_{6}F_{5}\left(\begin{array}{c}
-k,1-k-\beta_{1},1-k-\beta_{2},1-k-\beta_{3},a_{1},a_{2}\\
1-k-\alpha_{1},1-k-\alpha_{2},b_{1},b_{2},b_{3}
\end{array};\frac{c}{d}\right).\nonumber 
\end{align}
Therefore our desired hypergeometric product is
\begin{align*}
\sum_{n\ge0}\frac{z^{2n}}{n!}\left(\frac{i}{4}\right)^{n}&\frac{\left(a\right)_{n}\left(b\right)_{n}}{\left(a+b\right)_{n}\left(\frac{a+b}{2}\right)_{n}\left(\frac{a+b+1}{2}\right)_{n}}\\
&\times{}_{6}F_{5}\left(\begin{array}{c}
-n,1-n-a-b,1-n-\frac{a+b}{2},1-n-\frac{a+b+1}{2},a,b\\
1-n-a,1-n-b,a+b,\frac{a+b}{2},\frac{a+b+1}{2}
\end{array};-1\right).
\end{align*}
Notice that the odd power terms in this sum are equal to zero so that our sum can be written as
\begin{align*}
\sum_{n\ge0}\frac{z^{4n}}{\left(2n\right)!}\frac{\left(-1\right)^{n}}{4^{2n}}&\frac{\left(a\right)_{2n}\left(b\right)_{2n}}{\left(a+b\right)_{2n}\left(\frac{a+b}{2}\right)_{2n}\left(\frac{a+b+1}{2}\right)_{2n}}\\
&\times{}_{6}F_{5}\left(\begin{array}{c}
-2n,1-2n-a-b,1-2n-\frac{a+b}{2},1-2n-\frac{a+b+1}{2},a,b\\
1-2n-a,1-2n-b,a+b,\frac{a+b}{2},\frac{a+b+1}{2}
\end{array};-1\right),
\end{align*}
which gives the desired result after applying some standard Pochhammer identities.
\end{proof}

To express the associated hypergeometric star MZVSs, we need to introduce the previously defined \textbf{hypergeometric Bernoulli numbers} $B_{n}^{\left(a,b\right)}$ \cite{Christophe2},
defined through their generating function
\[
\sum_{n\ge0}\frac{B_{n}^{\left(a,b\right)}}{n!}z^{n}=\frac{1}{{_{1}F_{1}\left(\begin{array}{c}
a\\
a+b
\end{array};z\right)}}.
\]
This parallels the appearance of Bernoulli numbers when classically expressing $\zeta(2n)$. We use now the generating function
\[
\prod_{k\ge1}\left(1-\frac{z^{2}}{z_{a,b;k}^{2}}\right)^{-1}=\sum_{n\ge0}\zeta_{a,b}^{*}\left(\left\{ 2\right\} ^{n}\right)z^{2n}.
\]
By \eqref{eq:infinite product}, this is also
\[
\prod_{k\ge1}\left(1-\frac{z^{2}}{z_{a,b;k}^{2}}\right)^{-1}=\frac{1}{_{1}F_{1}\left(\begin{array}{c}
a\\
a+b
\end{array};z\right){}_{1}F_{1}\left(\begin{array}{c}
a\\
a+b
\end{array};-z\right)}.
\]
Re-expressing in terms of hypergeometric Bernoulli polynomials, we see
\begin{align*}
\prod_{k\ge1}\left(1-\frac{z^{2}}{z_{a,b;k}^{2}}\right)^{-1} & =\sum_{k,l\ge0}\frac{B_{k}^{\left(a,b\right)}B_{l}^{\left(a,b\right)}}{k!l!}\left(-1\right)^{k}z^{k+l}\\
 & =\sum_{n\ge0}\frac{z^{n}}{n!}\sum_{k=0}^{n}\binom{n}{k}\left(-1\right)^{k}B_{k}^{\left(a,b\right)}B_{n-k}^{\left(a,b\right)}.
\end{align*}
We deduce
\begin{equation}
\label{zeta*ab2n}
\zeta_{a,b}^{*}\left(\left\{ 2\right\} ^{n}\right)=\sum_{k=0}^{2n}\binom{2n}{k}\left(-1\right)^{k}B_{k}^{\left(a,b\right)}B_{2n-k}^{\left(a,b\right)}.
\end{equation}
For example,
\[
\zeta_{a,b}^{*}\left(\left\{ 2\right\} \right)=\frac{ab}{\left(a+b\right)^{2}\left(1+a+b\right)}=\zeta_{a,b}\left(\left\{ 2\right\} \right)
\]
and
\[
\zeta_{a,b}^{*}\left(\left\{ 2\right\} ^{2}\right)=\frac{ab\left(-a^{2}-a^{3}-b^{2}-b^{3}+ab\left(a+b+5\right)\left(a+b+2\right)\right)}{2(a+b)^{4}(1+a+b)^{2}(2+a+b)(3+a+b)}.
\]
Although defining and using hypergeometric Bernoulli numbers may seem here like a trivial rewriting, since it essentially defined these numbers as the coefficients of the generating product of the multiple zeta star values, this approach has the major advantage that Bernoulli numbers satisfy linear recurrences that facilitate their fast compuation. Given these easily computable Bernoulli numbers, we can then in turn calculate MZSVs. More precisely,
the hypergeometric Bernoulli numbers satisfy the linear recurrence
\[
\sum_{k=0}^{n} \binom{a+b+n-1}{k} \binom{a-1+n-k}{n-k} B_{k}^{\left( a,b \right)} = \left(a+b\right)_{n} \delta_{n},
\]
with initial condition $B_{0}^{\left( a,b \right)}=1$ \cite{Christophe2}. This recurrence is a trivial consequence of the identity \eqref{hprod}
\[
H\left( t \right)E\left( -t \right)=1
\]
and allows us to explicitly compute $B_n^{\left( a,b \right)}$ in terms of the lower-order numbers $B_{n-k}^{\left( a,b \right)}$, and in turn to deduce the values of $\zeta_{a,b}^{*}\left(\left\{ 2\right\} ^{n}\right)$ using \eqref{zeta*ab2n}.
The same technique will be used in the next sections, and gives a practical approach whenever MZV evaluations are significantly easier that those for MZSVs.

\section{Bessel zeta}\label{section4}
\subsection{Some multiple Bessel zeta values}
We apply our above sum/product methodology to the case of multiple Bessel zeta values. Consider 
\begin{equation}\label{eq:j_nu}
j_\nu(z):=2^\nu \Gamma(\nu+1) \frac{J_\nu(z)}{z^\nu},
\end{equation}
where $J_\nu$ is the Bessel function of the first kind of order $\nu$.
The function $j_\nu(z)$ is a normalized version of $J_{\nu}$ that satisfies $j_\alpha(0)=1$.
It has pairs of zeros $\{\pm z_{\nu,k},\,\ k=1,2\dots\}, $ where $z_{\nu,k}$ denotes the $k-$th zero with positive real part, the zeros being numbered in order of  their real parts. Then we have the Weierstrass factorization \cite[(8.544, Page 942)]{Table}
\begin{equation}
\label{factorization}
j_\nu(z) = \prod_{k=1}^\infty \left(1- \frac{z^2}{z^2_{\nu,k}}\right),\,\, \forall z \in \mathbb{C},
\end{equation}
along with the Bessel zeta function
\begin{equation}
\zeta_{B,\nu}(s) := \sum_{n=1}^\infty \frac{1}{z_{\nu,n}^s}.
\end{equation}
We notice that the zeta function is built from half the zeros of the function $j_{\nu}$ only, in the same way as the Riemann zeta  function is built from half of the zeros of the function $\frac{\sin \pi x}{\pi x}$.

Frappier \cite{Frappier} developed an extensive theory centered around these functions. This led him to define a family of Bernoulli polynomials $B_{n,\nu}\left(x\right)$ called ``$\alpha$-Bernoulli polynomials" \cite[(Equation 1)]{Frappier} by the generating function
\begin{equation}
\frac{e^{(x-\frac12)z}}{j_\nu(\frac{iz}{2})} = \sum_{n=0}^\infty \frac{B_{n,\nu}(x)}{n!}z^n.
\label{Bessel Bernoulli polynomials}
\end{equation}
Letting $x=\frac12$ and scaling $z$ in this definition gives
\begin{equation}\label{3.1}
\frac{1}{j_\nu(z)} = \prod_{k=1}^\infty \left(1- \frac{z^2}{z^2_{\nu,k}}   \right)^{-1} =   \sum_{n=0}^\infty \frac{B_{n,\nu}(\frac12)}{n!}\left(\frac{2z}{i}\right)^n.
\end{equation}
Mapping $z^2 \mapsto -z^2$ and using the series expansion for $J_\nu(x)$ yields the sum/product representation
\begin{equation}\label{3.2}
j_\nu(iz) = \prod_{k=1}^\infty \left(1+ \frac{z^2}{z^2_{\nu,k}}   \right) = \sum_{k=0}^{\infty}  \frac{\Gamma(\nu+1)}{k!\Gamma(\nu+k+1)}\left(\frac{z}{2}\right)^{2k}.
\end{equation}
\begin{remark}
The case $\nu=\frac{1}{2}$ recovers the Riemann zeta function
since
\[
j_{\frac{1}{2}}\left(z\right)=\frac{\sin z}{z}=\prod_{k\ge1}\left(1-\frac{z^{2}}{k^{2}\pi^{2}}\right),
\]
so that the zeros, the zeta function and the Bernoulli polynomials are
\[
z_{\frac{1}{2},k}=k\pi,\thinspace\thinspace\zeta_{B,\frac{1}{2}}\left(s\right)=\frac{1}{\pi^{s}}\zeta\left(s\right),
\thinspace\thinspace
B_{n,\frac{1}{2}}\left(z\right)=B_{n}\left(z\right).
\]
The case $\nu=-\frac{1}{2}$ corresponds to
\[
j_{-\frac{1}{2}}\left(z\right)=\cos z=\prod_{k\ge1}\left(1-\frac{z^{2}}{\left(k+\frac{1}{2}\right)^{2}\pi^{2}}\right),
\]
so that
\[
z_{-\frac{1}{2},k}=\left(k+\frac{1}{2}\right)\pi,\thinspace\thinspace\zeta_{B,-\frac{1}{2}}\left(s\right)=\frac{2^{s}-1}{\pi^{s}}\zeta\left(s\right)
\]
and the Bernoulli polynomials $B_{n,-\frac{1}{2}}\left(z\right)$ coincide with the usual Euler polynomials $E_{n}\left(z\right)$
defined by generating function
\begin{equation}\label{eulerpoly}
\sum_{n\ge0}\frac{E_{n}\left(z\right)}{n!}x^{n}=\frac{2e^{xz}}{e^{x}+1}.
\end{equation}
\end{remark}
We will then heavily depend on the symmetrized averages of zeta functions. Define
\begin{equation}
S_G(mn,k) = \sum_{|\textbf{a}|=n} \z(m a_1, \ldots, m a_k),
\end{equation}
and
\begin{equation}
S_G^\ast(mn,k) =\sum_{|\textbf{a}|=n} \zs(m a_1, \ldots, m a_k),
\end{equation}
which are averages over all multiple zeta (star) functions with weight $mn$ and depth $k$.
Hoffman \cite{Hoffman1} has shown that, in the case of the Riemann multiple zeta function, the sum $S(2n,k)$ has a closed form evaluation. Similar results have also been found for renormalized Hurwitz zeta functions \cite{Hoffman2, Chung1}. Analogously to how a quantum zeta function may have a nice closed form while individual energy eigenstates do not, the averaged $S_G(mn,k)$ might have a nice closed form while individual multiple zeta values do not \cite{Chen1}. 

The averages have the following generating functions
(see \cite{Hoffman1}), for which we provide a proof for the sake of completeness.
\begin{thm}
\label{thm1}
The following generating products hold:
\begin{align}
\prod_{k=1}^{\infty}\frac{ \left(1+\frac{(y-1)t}{a_k^s}  \right) }{\left(1-\frac{t}{a_k^s}  \right) } &=  \sum_{n=0}^{\infty}\sum_{k=0}^{n}S_G(sn,k)y^k t^n,  \label{2.1c} \\
\prod_{k=1}^{\infty}\frac{ \left(1-\frac{t}{a_k^s}  \right)}{ \left(1-\frac{(y+1)t}{a_k^s}  \right)} &= \sum_{n=0}^{\infty}\sum_{k=0}^{n}S_G^\ast(sn,k)y^k t^n. \label{2.1d}
\end{align}
\begin{proof}
When we take $x_i = \sum_{k=1}^{\infty} \frac{t^k}{a_i^{ks}}$, we have $e_n = \sum_{n=k}^{\infty} S_G(sn,k) t^n$, since the coefficient of each $t_n$ term will sum over $\left(  a_{n_1}^{k_1 s}  a_{n_2}^{k_2 s}\cdots   \right)^{-1}$ such that $\sum_i k_i = n$. This is precisely $S_G(sn,k)$, which yields
\begin{equation}
 \prod_{k=1}^{\infty}\frac{ \left(1+\frac{(y-1)t}{a_k^s}  \right) }{\left(1-\frac{t}{a_k^s}  \right) } = \prod_{i=1}^{\infty}\left(1+y\left(\frac{t}{a_i^s}+\frac{t^2}{a_i^{2s}}+\ldots \right) \right) =   \sum_{k=0}^{\infty}e_k y^k = \sum_{k=0}^{\infty}\sum_{n=k}^{\infty}S_G(sn,k)y^k t^n.
\end{equation}
Reversing the order of summation proves \eqref{2.1c}. This same choice of $x_i$ will yield $h_n =  \sum_{n=k}^{\infty} S_G^\ast(sn,k) t^n$ since we allow equality between the $n_i$. Using the relation $H(t)=E(-t)^{-1}$ we have
\begin{equation}
\sum_{k=0}^{\infty}\sum_{n=k}^{\infty}S_G^\ast(sn,k)y^k t^n = \sum_{k=0}^{\infty}h_k y^k =E(-y)^{-1} = \prod_{k=1}^{\infty}\frac{ \left(1-\frac{t}{a_k^s}  \right)}{ \left(1-\frac{(y+1)t}{a_k^s}  \right)}.
\end{equation}
Reversing the order of summation on the left hand side completes the proof.
\end{proof}
\end{thm}

Using these product expansions, we can then apply the general Theorem \ref{thm1} to the Bessel zeta function
along with the corresponding $\zeta_{B,\nu}^\ast$, $S_{B,\nu}$, and $S_{B,\nu}^\ast$.

\begin{thm}
\label{Thm8}
The following evaluations hold for the zeta function built out of Bessel zeros:
\begin{align}
\zeta_{B,\nu}(\{2\}^n) &=  \frac{1}{2^{2n}n!(\nu+1)_n}, \label{zhang1}\\
\zeta_{B,\nu}^\ast(\{2\}^n) &=  \frac{B_{2n,\nu}(\frac12)}{(2n)!}2^{2n}(-1)^n,\\
S_{B,\nu}(2n,k) &= \sum_{r=k}^n (-1)^{n-k}\binom{r}{k}\frac{ 2^{2n-4r} \Gamma(\nu+1)}{r!(2n-2r)!\Gamma(\nu+r+1)}B_{2n-2r,\nu}\left(\frac12\right), \label{bessel1}\\
S_{B,\nu}^\ast(2n,k) &= (-1)^{n} \sum_{r=k}^n \binom{r}{k}\frac{\Gamma(\nu+1)}{  2^{2n-4r} (n-r)!(2r)!\Gamma(\nu+n-r+1)}B_{2r,\nu}\left(\frac12\right).
\end{align}
\begin{proof}
The results follow from a straightforward application of Theorem \ref{thm1} and Lemma \ref{lemma1} with the product expansions \eqref{3.1} and \eqref{3.2}. Note that \eqref{zhang1} was previously derived in \cite[p. 355]{Zhang}. Because of their dependence on the Bessel-Bernoulli numbers, the other three identities appear to be new. 
\end{proof}
\end{thm}
We note that \eqref{bessel1} generalizes \cite[(Theorem 3)]{Hoffman1}, which is the case $\nu=\frac12$. Also, this special case 
$\nu=\frac12$ is proved using a limiting argument in a Hurwitz-like generating function for $S\left(2n,k\right)$ in \cite{Gencev}.

\begin{thm}
The following evaluation holds:
\begin{equation}
\zeta_{B,\nu}(\{4\}^n) =  \frac{ 1}  {2^{4n}n!\left(\nu+1\right)_{2n} \left(\nu+1\right)_n}.
\end{equation}
\begin{proof}
We dissect the generating product as follows:
\begin{align*}
\sum_{n=0}^\infty (-1)^n \zeta_{B,v}(\{4\}^n)t^{2n} &= \prod_{k=1}^\infty \left(1 - \frac{t^2}{z_{\nu,k}^4} \right) \\
&= \prod_{k=1}^\infty \left(1 - \frac{t}{z_{\nu,k}^2} \right)\prod_{k=1}^\infty \left(1 + \frac{t}{z_{\nu,k}^2} \right) \\
&= \sum_{n=0}^\infty (-1)^n \zeta_{B,v}(\{2\}^n)t^{n}  \sum_{n=0}^\infty\zeta_{B,v}(\{2\}^n)t^{n}.
\end{align*}
Comparing coefficients of $t^n$, while implicitly regularizing $\zeta_{B,\nu} (\{2\}^0)=1$, yields
\begin{align*}
\zeta_{B,\nu}(\{4\}^n) &=  (-1)^{n}  \sum_{l=0}^{2n} (-1)^{l} \zeta_{B,\nu}(\{2\}^{l})\zeta_{B,\nu}(\{2\}^{2n-l}) ) \\
&= (-1)^{n}   \sum_{l=0}^{2n} (-1)^{l} \frac{\Gamma^2(\nu+1)}{2^{4n}(2n-l)!l!\Gamma(\nu+l+1)\Gamma(\nu+2n-l+1)}.
\end{align*}
This sum is now identified in terms of a Gauss hypergeometric function 
\begin{align*}
\zeta_{B,\nu}(\{4\}^n) &=   \frac{(-1)^{n} \Gamma(\nu+1)}{2^{4n}(2n)!\Gamma(2n+\nu+1)}\hyp21{-2n,-2n-\nu}{\nu+1}{-1},
\end{align*}
which can be evaluated using Kummer's identity \cite[7.3.6.2]{Prudnikov} as
\begin{equation}
\hyp21{-2n,-2n-\nu}{\nu+1}{-1} =  (-1)^n\frac{2\Gamma(\nu+1)\Gamma(2n)}{\Gamma(n+\nu+1)\Gamma(n)}.
\end{equation}
This gives the partial result
\begin{align*}
\zeta_{B,\nu}(\{4\}^n) &= \frac{ \Gamma^2(\nu+1)}  {2^{4n}n!\Gamma(2n+\nu+1) \Gamma(n+\nu+1)}.
\end{align*}
Simplifying this expression yields the final result.
\end{proof}
\end{thm}

\subsection{Krein's expansion and an extension of Euler's identity}

Using a recent result from \cite{Sherstyukov}, we can relate the
Bessel-Bernoulli polynomials to an alternate Bessel zeta function, thereby developing the theory of Bessel zeta functions.
This also provides intuition for the `correct' generalization of the Bessel zeta function to the multiple zeta case -- as with $q$-series, the most useful generalization is often non-intuitive.
\begin{thm}
Define the alternate Bessel zeta function $\tilde{\zeta}_{\nu}$ as
\begin{equation}\label{modified zeta}
\tilde{\zeta}_{\nu}\left(r\right)=\sum_{k\ge1}\frac{1}{j_{\nu+1}\left(z_{\nu,k}\right)z_{\nu,k}^{r+2}}.
\end{equation}
Then for $n\ge\left[\frac{\nu}{2}+\frac{1}{4}\right]+1,$ the Bessel-Bernoulli
polynomials can be expressed as 
\begin{equation}
\frac{\left(-1\right)^{n}2^{2n}}{2n!}B_{2n,\nu}\left(\frac{1}{2}\right)=4\left(\nu+1\right)\tilde{\zeta}_{\nu}\left(2n\right).\label{eq:BZeta}
\end{equation}
The case $\nu=\frac{1}{2}$ recovers Euler's identity
\[
\frac{\left(2\pi\right)^{2n}}{2\left(2n\right)!}\left(-1\right)^{n-1}B_{2n}=\zeta\left(2n\right).
\]
\end{thm}

\begin{proof}
First define
\[
p:=\left[\frac{\nu}{2}+\frac{1}{4}\right]+1,
\]
where square brackets denote the integer part. We then have Krein's expansion \cite{Sherstyukov}
\[
\frac{1}{J_{\nu}\left(z\right)}=\frac{P_{p}\left(z\right)}{z^{\nu}}-2z^{2p-\nu}\sum_{k\ge1}\frac{1}{J_{\nu+1}\left(z_{\nu,k}\right)z_{\nu,k}^{2p-\nu-1}\left(z^{2}-z_{\nu.k}^{2}\right)},
\]
where $P_{p}\left(z\right)$ is the polynomial of degree $2p-2$ defined as the truncated Taylor expansion at $0$ of $\frac{z^{\nu}}{J_{\nu}\left(z\right)}$:
\[
P_{p}\left(z\right)=\sum_{m=0}^{p-1}\frac{d^{2m}}{dz^{2m}}\left(\frac{z^{\nu}}{J_{\nu}\left(z\right)}\right)\Biggr|_{ z=0}\frac{z^{2m}}{2m!}.
\]
Now take $\vert z\vert<\vert z_{\nu,1}\vert,$ the smallest zero of
$J_{\nu},$ so that
\[
\frac{1}{z_{\nu,k}^{2p-\nu-1}\left(z^{2}-z_{\nu.k}^{2}\right)}=-\frac{1}{z_{\nu,k}^{2p-\nu+1}\left(1-\frac{z^{2}}{z_{\nu.k}^{2}}\right)}=-\sum_{q\ge0}\frac{z^{2q}}{z_{\nu,k}^{2q+2p-\nu+1}}
\]
and the infinite sum in Krein's expansion is
\begin{align*}
-2z^{2p-\nu}\sum_{k\ge1}\frac{1}{J_{\nu+1}\left(z_{\nu,k}\right)z_{\nu,k}^{2p-\nu-1}\left(z^{2}-z_{\nu.k}^{2}\right)} & =2\sum_{q\ge0}z^{2p+2q-\nu}\sum_{k\ge1}\frac{1}{J_{\nu+1}\left(z_{\nu,k}\right)z_{\nu,k}^{2q+2p-\nu+1}}.
\end{align*}
Using the normalization
\[
J_{\nu+1}\left(z\right)=\frac{z^{\nu+1}}{2^{\nu+1}\Gamma\left(\nu+2\right)}j_{\nu+1}\left(z\right),
\]
this sum can be expressed as
\[
2^{\nu+2}\Gamma\left(\nu+2\right)\sum_{q\ge0}z^{2p+2q-\nu}\sum_{k\ge1}\frac{1}{j_{\nu+1}\left(z_{\nu,k}\right)z_{\nu,k}^{2q+2p+2}}.
\]
Using the definition of the alternate Bessel zeta function \eqref{modified zeta}, we obtain
\[
\frac{1}{j_{\nu}\left(z\right)}=\frac{P_{p}\left(z\right)}{2^{\nu}\Gamma\left(\nu+1\right)}+4\left(\nu+1\right)\sum_{q\ge0}z^{2p+2q}\tilde{\zeta}_{\nu}\left(2q+2p\right).
\]
Identifying with \eqref{3.1}, we deduce \eqref{eq:BZeta}.

In the case $\nu=\frac{1}{2},$ we have
$
j_{\frac{3}{2}}\left(z\right)=\frac{3}{z^{2}}\left(\frac{\sin z}{z}-\cos z\right)
$
and the zeros $z_{\frac{1}{2},k}=k\pi$, so that
\[
z_{\frac{1}{2},k}^{r+2}j_{\frac{3}{2}}\left(z_{\frac{1}{2},k}\right)=
3\left(-1\right)^{k-1}\left(k\pi\right)^{r}
\]
and the alternate Bessel zeta function reads
\[
\tilde{\zeta}_{\frac{1}{2}}\left(r\right)
=\frac{1}{3\pi^{r}}\frac{2^{r+2}-8}{2^{r+2}}\zeta\left(r\right).
\]
As a result, we have the reduction
\begin{align*}
4\left(\nu+1\right)\tilde{\zeta}_{\frac{1}{2}}\left(2n\right) & =\frac{1}{\pi^{2n}}\frac{2^{2n-1}-1}{2^{2n-2}}\zeta\left(2n\right).
\end{align*}
Moreover, the Bessel Bernoulli numbers \eqref{Bessel Bernoulli polynomials}
are related to the usual Bernoulli numbers as
\[
\frac{B_{2n,\frac{1}{2}}}{2n!}=\left(-1\right)^{n}2^{2n}\frac{B_{2n}\left(\frac{1}{2}\right)}{2n!}=\left(-1\right)^{n}2^{2n}\frac{ \left(2^{1-2n}-1\right)B_{2n} }{2n!}.
\]
Therefore, in the case $\nu=\frac{1}{2},$ the identity \eqref{eq:BZeta}
reduces to Euler's identity
\[
\frac{B_{2n}}{2n!}=\frac{2\left(-1\right)^{n-1}}{\left(2\pi\right)^{2n}}\zeta\left(2n\right).
\]
\end{proof}

\begin{remark}
The introduction of the alternate Bessel zeta function \eqref{modified zeta} can be justified by the fact that
\begin{equation}
\label{zeta*zetatilde}
\zeta_{B,\nu}^\ast(\{2\}^n) = 4\left( \nu +1 \right) \tilde{\zeta}_{\nu}\left(2n\right),
\end{equation}
obtained by comparing the expression of $\zeta_{B,\nu}^\ast(\{2\}^n)$ in Thm. \ref{Thm8} and in equation \eqref{eq:BZeta}.
The function on the left of \eqref{zeta*zetatilde} is a multiple zeta value, hence a multiple nested sum, while the zeta function on the right is a simple sum; this simplification comes at the price of evaluating the Bessel function $j_{\nu+1}$ at the roots of the contiguous function $j_{\nu}$. 

Moreover, notice that the recurrence on Bessel functions
\[
J_{\nu -1}\left( z \right) +J_{\nu +1}\left( z \right) = \frac{2\nu}{z}J_{\nu}\left( z \right)
\]
implies that 
\[
j_{\nu+1}\left( z_{\nu,k} \right) = -j_{\nu-1}\left( z_{\nu,k} \right),
\]
so that either $j_{\nu+1}$ or $j_{\nu-1}$ can be used in the definition of the alternate Bessel zeta function \eqref{modified zeta}.

Finally, the alternate Bessel zeta function \eqref{modified zeta} appears in the paper by I.N. Sneddon  \cite{Sneddon}, where a recurrence relation and several special values can be found.
\end{remark}

\subsection{A Bessel-Gessel-Viennot identity}\label{besselgessel}
We now begin an in-depth study of the average $S_{B,\nu}(2n,k)$. While the individual zeta functions it sums over do not have closed form expressions, we can deduce closed form expressions for the average. For $\nu = -\frac12$, we deduce pretty results for multiple zeta values built from only odd integers \cite{Hoffman2}. In \cite{Hoffman1}, Hoffman uses generating functions to express the average $S(2n,k)$ in two different ways (in the case $\nu = \frac12$). This yields \cite[(Theorem 1)]{Hoffman1}
\begin{equation}
S\left(2n,k\right)=\sum_{j=0}^{\left\lfloor \frac{k-1}{2}\right\rfloor }\frac{\left(i\pi\right)^{2j}\zeta\left(2n-2j\right)}{2^{2k-2j-2}\left(2j+1\right)!}\binom{2k-2j-1}{k}
,\thinspace\thinspace k \le n,\label{eq:thm1}
\end{equation}
on one side and 
\begin{equation}
S\left(2n,k\right)=\frac{\left(-1\right)^{n-k}}{\left(2n+1\right)!}\sum_{i=0}^{n-k}\binom{n-i}{k}\binom{2n+1}{2i}2^{2i}B_{2i}\left(\frac{1}{2}\right)
,\thinspace\thinspace k \le n,\label{eq:thm1_1}
\end{equation}
on the other side. Equating these two expressions produces the identity
\begin{equation}
\label{Gessel-Viennot range1}
\sum_{i=0}^{\left\lfloor \frac{k-1}{2}\right\rfloor }\binom{2k-2i-1}{k}\binom{2n+1}{2i+1}B_{2n-2i}= \frac{\left(-1\right)^{n-k}}{\left(2n+1\right)!}\sum_{i=0}^{n-k}\binom{n-i}{k}\binom{2n+1}{2i}2^{2i}B_{2i}\left(\frac{1}{2}\right)
,\thinspace\thinspace k \le n,
\end{equation}
a variation of the Gessel-Viennot identity \cite{Gessel Viennot}
\begin{equation}
\label{Gessel-Viennot range2}
\sum_{i=0}^{\left\lfloor \frac{k-1}{2}\right\rfloor }\binom{2k-2i-1}{k}\binom{2n+1}{2i+1}B_{2n-2i}=\frac{2n+1}{2}\binom{2k-2n}{k}
,\thinspace\thinspace k > n,
\end{equation}
that is valid on the complementary range $k\leq n.$

Mirroring Hoffman \cite[(Theorem 1)]{Hoffman1} and using a method inspired by \cite{Ding}, we derive the extension of identities \eqref{eq:thm1} and \eqref{eq:thm1_1} to the Bessel case with an arbitrary value of $\nu$ and deduce a one-parameter generalization of the identity \eqref{Gessel-Viennot range1}. As a byproduct of our proof, we deduce a generalization of the original Gessel-Viennot identity to arbitrary $\nu$.

First we follow Hoffman's approach to derive the following result:
\begin{thm}
The sum $S_{B,\nu}(2n,k)$ can be expressed as 
\begin{equation}
S_{B,\nu}(2n,k) = \sum_{r=k}^n (-1)^{n-k}\binom{r}{k}\frac{ 2^{2n-4r} \Gamma(\nu+1)}{r!(2n-2r)!\Gamma(\nu+r+1)}B_{2n-2r,\nu}\left(\frac12\right),
\,\, k\le n,
\end{equation}
which is equation \eqref{bessel1}, or alternatively as
\begin{equation}
\label{second expression}
S_{B,\nu}(2n,k)=\frac{1}{k!}\sum_{j=0}^{\left\lfloor \frac{k-1}{2}\right\rfloor }
\frac{\left(-1\right)^{j}}{2^{2j}}
\binom{k-1-j}{j}
\frac{\Gamma\left(\nu+k-j\right)}{\Gamma\left(\nu+1+j\right)}
\zeta_{B,\nu}\left(2n-2j\right), \,\, k\le n.
\end{equation}
\end{thm}
\begin{proof}
We only need prove \eqref{second expression}.
Following Hoffman's approach and beginning with the generating product \eqref{2.1d}, we rewrite the generating function
\begin{equation}
\mathcal{F}\left(t,y\right)=\sum_{n=0}^\infty\sum_{k=0}^{n}S_{B,\nu}(2n,k)y^{k}t^{n}=\prod_{k=1}^\infty\frac{\left(1+\frac{\left(y-1\right)t}{z_{\nu,k}^{2}}\right)}{\left(1-\frac{t}{z_{\nu,k}^{2}}\right)}=\frac{j_{\nu}\left(\sqrt{t\left(1-y\right)}\right)}{j_{\nu}\left(\sqrt{t}\right)},
\end{equation}
so that
\begin{equation}
\mathcal{F}\left(t,y\right)=\sum_{k=0}^\infty y^{k}G_{k}\left(t\right)
\end{equation}
with
\begin{equation}
G_{k}\left(t\right)=\frac{1}{j_{\nu}\left(\sqrt{t}\right)}\frac{\left(-t\right)^{k}}{k!}\frac{d^{k}}{dt^{k}}j_{\nu}\left(\sqrt{t}\right) = \frac{t^{k}}{k!}\frac{1}{2^{2k}\left(\nu+1\right)_{k}}\frac{j_{\nu+k}\left(\sqrt{t}\right)}{j_{\nu}\left(\sqrt{t}\right)}.
\end{equation}
This identity is a consequence of the laddering operation \cite[1.10.1.5]{Brychkov}
\begin{equation}\label{laddering}
\frac{d^{k}}{dt^{k}}\frac{J_{\nu}\left(\sqrt{t}\right)}{\left(\sqrt{t}\right)^{\nu}}=\left(-\frac{1}{2}\right)^{k}\frac{J_{\nu+k}\left(\sqrt{t}\right)}{\left(\sqrt{t}\right)^{\nu+k}}, 
\end{equation}
and the definition of $j_\nu$ given in \eqref{eq:j_nu}. 

We can then exploit the reduction of order formula
\begin{equation}
j_{\nu+m}\left(\sqrt{z}\right)=\frac{R_{m-1,\nu+1}\left(\sqrt{z}\right)}{\left(\frac{\sqrt{z}}{2}\right)^{m-1}}\left(\nu+2\right)_{m-1}j_{\nu+1}\left(\sqrt{z}\right)-\frac{R_{m-2,\nu+2}\left(\sqrt{z}\right)}{\left(\frac{\sqrt{z}}{2}\right)^{m}}\left(\nu+1\right)_{m}j_{\nu}\left(\sqrt{z}\right)\label{eq:Lommel},
\end{equation}
where $R_{m,\nu}\left(z\right)$ are the Lommel polynomials \cite[p.294]{Watson}. This is a direct consequence of the corresponding result for Bessel $J_{\nu}$ functions \cite[Eq. 1 p. 295]{Watson} and the normalization \eqref{eq:j_nu}.
Using \eqref{eq:Lommel} gives
\begin{align*}
G_{k}\left(t\right) & =\frac{1}{2^{2k}\left(\nu+1\right)_{k}}\frac{t^{k}}{k!}\left[\frac{R_{k-1,\nu+1}\left(\sqrt{t}\right)}{\left(\frac{\sqrt{t}}{2}\right)^{k-1}}\left(\nu+2\right)_{k-1}\frac{j_{\nu+1}\left(\sqrt{t}\right)}{j_{\nu}\left(\sqrt{t}\right)}-\frac{R_{k-2,\nu+2}\left(\sqrt{t}\right)}{\left(\frac{\sqrt{t}}{2}\right)^{k}}\left(\nu+1\right)_{k}\right]\\
 & =\frac{1}{2^{2k}}\frac{t^{k}}{k!}\left[\frac{R_{k-1,\nu+1}\left(\sqrt{t}\right)}{\left(\frac{\sqrt{t}}{2}\right)^{k-1}}\frac{4}{t}\mathcal{Z}_{\nu}\left(\sqrt{t}\right)-\frac{R_{k-2,\nu+2}\left(\sqrt{t}\right)}{\left(\frac{\sqrt{t}}{2}\right)^{k}}\right].
\end{align*}
Inserting the explicit expression for the Lommel polynomials \cite[Eq. 3 p. 296]{Watson},
\[
R_{m,\nu}\left(z\right)=\sum_{j=0}^{\left\lfloor \frac{m}{2}\right\rfloor }\left(-1\right)^{j}\binom{m-j}{j}\frac{\Gamma\left(\nu+m-j\right)}{\Gamma\left(\nu+j\right)}\left(\frac{z}{2}\right)^{2j-m}
\]
yields
\begin{align*}
G_{k}\left(t\right) 
 & =\frac{1}{k!}\left(\mathcal{Z}_{\nu}\left(t\right)\sum_{j=0}^{\left\lfloor \frac{k-1}{2}\right\rfloor }\left(-1\right)^{j}\binom{k-1-j}{j}\frac{\Gamma\left(k+\nu-j\right)}{\Gamma\left(\nu+1+j\right)}\left(\frac{\sqrt{t}}{2}\right)^{2j}\right)\\
 & -\frac{1}{k!}\left(\sum_{j=0}^{\left\lfloor \frac{k-2}{2}\right\rfloor }\left(-1\right)^{j}\binom{k-2-j}{j}\frac{\Gamma\left(k+\nu-j\right)}{\Gamma\left(\nu+2+j\right)}\left(\frac{\sqrt{t}}{2}\right)^{2j+2}\right).
\end{align*}
We now need to compute the coefficient of $t^{n}$ in this expression; by
definition, this is $S_{B,\nu}\left(2n,k\right).$
The first term is expanded as
\begin{align*}
\mathcal{Z}_{\nu}\left(t\right)\sum_{j=0}^{\left\lfloor \frac{k-1}{2}\right\rfloor }\left(-1\right)^{j}&\binom{k-1-j}{j}\frac{\Gamma\left(k+\nu-j\right)}{\Gamma\left(\nu+1+j\right)}\left(\frac{\sqrt{t}}{2}\right)^{2j}  \\
& =\sum_{q\ge1}\sum_{j=0}^{\left\lfloor \frac{k-1}{2}\right\rfloor }\zeta_{B,\nu}\left(2q\right)t^{q}\left(-1\right)^{j}\binom{k-1-j}{j}\frac{\Gamma\left(k+\nu-j\right)}{\Gamma\left(\nu+1+j\right)}\left(\frac{\sqrt{t}}{2}\right)^{2j}\\
 & =\sum_{n\ge1}t^{n}\sum_{j\ge0}\zeta_{B,\nu}\left(2n-2j\right)\left(-1\right)^{j}\binom{k-1-j}{j}\frac{\Gamma\left(k+\nu-j\right)}{\Gamma\left(\nu+1+j\right)}\frac{1}{2^{2j}},
\end{align*}
so that the coefficient of $t^{n}$ is identified as
\[
\begin{cases}
\sum_{j=0}^{n-1}\zeta_{B,\nu}\left(2n-2j\right)\left(-1\right)^{j}\binom{k-1-j}{j}\frac{\Gamma\left(k+\nu-j\right)}{\Gamma\left(\nu+1+j\right)}\frac{1}{2^{2j}} & 1\le n\le\left\lfloor \frac{k+1}{2}\right\rfloor ,\\
\sum_{j=0}^{\left\lfloor \frac{k-1}{2}\right\rfloor }\zeta_{B,\nu}\left(2n-2j\right)\left(-1\right)^{j}\binom{k-1-j}{j}\frac{\Gamma\left(k+\nu-j\right)}{\Gamma\left(\nu+1+j\right)}\frac{1}{2^{2j}} & n>\left\lfloor \frac{k+1}{2}\right\rfloor .
\end{cases}
\]
In the second term, we identify the coefficient of $t^{n}$ as
\[
\begin{cases}
\left(-1\right)^{n-1}\binom{k-1-n}{n-1}\frac{\Gamma\left(k+\nu-n+1\right)}{\Gamma\left(\nu+1+n\right)}\frac{1}{2^{2n}} & 1\le n\le\left\lfloor \frac{k}{2}\right\rfloor, \\
0 & \text{else.}
\end{cases}
\]
Therefore we obtain the following closed form for the value of $S_{\nu}\left(2n,k\right)$ as the coefficient of $t^{n}$ in $G_{k}\left(t\right)$:
\[
S_{\nu}\left(2n,k\right)=\sum_{j=0}^{\left\lfloor \frac{k-1}{2}\right\rfloor }\zeta_{B,\nu}\left(2n-2j\right)\left(-1\right)^{j}\binom{k-1-j}{j}\frac{\Gamma\left(k+\nu-j\right)}{\Gamma\left(\nu+1+j\right)}\frac{1}{2^{2j}},\thinspace\thinspace n\ge k,
\]
which is the desired result.
\end{proof}
We can also exploit the laddering operation used in the previous proof to find the generating function for Bessel zeta values, which will become useful later. 
\begin{prop}
The generating function for Bessel zeta values is
\begin{equation}
\mathcal{Z}_{\nu}\left(z\right):=\sum_{p=1}^\infty z^{p}\zeta_{B,\nu}\left(2p\right)=\frac{z}{4\left(\nu+1\right)}\frac{j_{\nu+1}\left(\sqrt{z}\right)}{j_{\nu}\left(\sqrt{z}\right)}.
\end{equation}
\end{prop}
\begin{proof}
This is due to the Weierstrass factorization \eqref{factorization}, from which we deduce
\begin{equation}
\log j_{\nu}\left(z\right)=-\sum_{p=1}^\infty\frac{z^{p}}{p}\zeta_{B,\nu}\left(2p\right),
\end{equation}
so that
\begin{equation}
-z\frac{\frac{d}{dz}j_{\nu}\left(\sqrt{z}\right)}{j_{\nu}\left(\sqrt{z}\right)}=\sum_{p=1}^\infty z^{p}\zeta_{B,\nu}\left(2p\right)=\mathcal{Z}_{\nu}\left(z\right).
\end{equation}
Using the laddering rule \eqref{laddering} gives the desired result.
\end{proof}
\begin{remark}

The case $\nu=\frac{1}{2}$ recovers Hoffman's identity \eqref{second expression}, while the case $\nu=-\frac{1}{2}$ yields the identities
\[
S(2n,k)=\frac{\left(-1\right)^{n-k}}{\left(2n\right)!}\sum_{r=k}^{n}\binom{r}{k}\binom{2n}{2r}2^{2n-2r}E_{2n-2r}\left(\frac{1}{2}\right),\,\ n \ge k,
\]
and
\[
S(2n,k)=\frac{(-1)^{n+1}2^{2n-2k+1}}{k(2n)!}\sum_{j=0}^{\left\lfloor \frac{k}{2}-\frac{1}{2}\right\rfloor }\left(2^{2n-2j}-1\right)B_{2n-2j}\binom{-2j+2k-2}{k-1}\binom{2n}{2j},\,\ n \ge k.
\]
Here $E_n\left( x \right)$ are the Euler polynomials defined by \eqref{eulerpoly}.
\end{remark}

\begin{cor}
This proof gives us as a by-product two additional identities; we know that $S_{\nu}\left(2n,k\right)=0$ for $1\le n\le k-1,$ so that
we can deduce
\begin{align}\label{eq:first_lommel}
\sum_{j=0}^{n-1}\zeta_{B,\nu}\left(2n-2j\right)\left(-1\right)^{j}&\binom{k-1-j}{j}\frac{\Gamma\left(k+\nu-j\right)}{\Gamma\left(\nu+1+j\right)}\frac{1}{2^{2j}}  \\
&=\left(-1\right)^{n-1}\binom{k-1-n}{n-1}\frac{\Gamma\left(k+\nu-n+1\right)}{\Gamma\left(\nu+1+n\right)}\frac{1}{2^{2n}}, \thinspace\thinspace1\le n\le\left\lfloor \frac{k}{2}\right\rfloor, \nonumber
\end{align}
and
\begin{equation}
\sum_{j=0}^{\left\lfloor \frac{k-1}{2}\right\rfloor }\zeta_{B,\nu}\left(2n-2j\right)\left(-1\right)^{j}\binom{k-1-j}{j}\frac{\Gamma\left(k+\nu-j\right)}{\Gamma\left(\nu+1+j\right)}\frac{1}{2^{2j}}=0,\thinspace\thinspace\left\lfloor \frac{k+1}{2}\right\rfloor <n\le k-1.\label{eq:second_lommel}
\end{equation}
\end{cor}
\begin{remark}
The second identity \eqref{eq:second_lommel} can be deduced from the orthogonality property
for Lommel polynomials as studied by Dickinson \cite{Dickinson} and
later rediscovered by Grosjean \cite[Eq. (61)]{Grosjean}. Consider
the discrete measure
\begin{align*}
\rho\left(x\right) & =\sum_{n\ge1}\frac{1}{z_{\nu-1,n}^{2}}\left[\delta\left(x-\frac{1}{z_{\nu-1,n}}\right)+\delta\left(x+\frac{1}{z_{\nu-1,n}}\right)\right].
\end{align*}
Then the Lommel polynomials $\left\{ R_{\nu,n}\left(\frac{1}{x}\right)\right\} $
are orthogonal with respect to this measure, i.e.
\[
\int_{-\infty}^{+\infty}R_{r,\nu}\left(\frac{1}{x}\right)R_{s,\nu}\left(\frac{1}{x}\right)\rho\left(x\right)dx=\frac{\delta_{r,s}}{2\left(\nu+r\right)}.
\]
Choosing $s\in\left[0,r\right]$ and expressing $x^{s}$ as a linear
combination of Lommel polynomials of degrees $\le s$ yields the orthogonality
property
\begin{equation}
\sum_{q=1}^{+\infty}\frac{R_{r,\nu}\left(z_{\nu-1,q}\right)}{z_{\nu-1,q}^{s+2}}=\frac{\Gamma\left(\nu\right)}{2^{r+2}\Gamma\left(\nu+r+1\right)}\delta_{r,s},\thinspace\thinspace0\le s\le r.\label{eq:orthogonality}
\end{equation}
Rewriting the left-hand side of \eqref{eq:second_lommel} by expanding the
zeta function as
\begin{align*}
\sum_{q\ge1}\sum_{j=0}^{\left\lfloor \frac{k-1}{2}\right\rfloor }&\frac{1}{z_{\nu,q}^{2n-2j}}\left(-1\right)^{j}\binom{k-1-j}{j}\frac{\Gamma\left(k+\nu-j\right)}{\Gamma\left(\nu+1+j\right)}\frac{1}{2^{2j}} 
\\ & =\sum_{q\ge1}\frac{1}{z_{\nu,q}^{2n}}\sum_{j=0}^{\left\lfloor \frac{k-1}{2}\right\rfloor }\left(-1\right)^{j}\binom{k-1-j}{j}\frac{\Gamma\left(k+\nu-j\right)}{\Gamma\left(\nu+1+j\right)}\left(\frac{z_{\nu,q}}{2}\right)^{2j} \\
&=\sum_{q\ge1}\frac{R_{k-1,\nu+1}\left(z_{\nu,q}\right)}{z_{\nu,q}^{2n-k+1}}
\end{align*}
shows that this is a special case of \eqref{eq:orthogonality} with $r=k-1$
and $s=2n-k-1;$ since $k\ge n,$ we have $s\le r$ so that $r-s=2k-2n\ne0$
for $\left\lfloor \frac{k}{2}\right\rfloor <n\le k-1$, hence this sum is equal to $0.$ The connection between MZVs and orthogonality relations is extremely unexpected, and deserves further study.
\end{remark}
\begin{remark}
Let us show that, in the case $\nu=\frac{1}{2}$, the first identity \eqref{eq:first_lommel} coincides with the Gessel-Viennot identity \eqref{Gessel-Viennot range2}: we use the relation 
$
\zeta_{B,\frac{1}{2}}\left(2n\right)=\frac{\zeta\left(2n\right)}{\pi^{2n}}=\frac{2^{2n-1}}{\left(2n\right)!}\left(-1\right)^{n-1}B_{2n},
$
and Euler's duplication formula
$
\Gamma\left(2z\right)=\frac{2^{2z-1}}{\sqrt{\pi}}\Gamma\left(z\right)\Gamma\left(z+\frac{1}{2}\right).
$
Substituting $\nu=\frac{1}{2}$ in the left-hand side of \eqref{eq:first_lommel}
gives, after some algebra,
\begin{align*}
\sum_{j=0}^{n-1}\zeta_{B,\frac{1}{2}}\left(2n-2j\right)\left(-1\right)^{j}&\binom{k-1-j}{j}\frac{\Gamma\left(k+\frac{1}{2}-j\right)}{\Gamma\left(\frac{1}{2}+1+j\right)}\frac{1}{2^{2j}}\\
&=\left(-1\right)^{n-1}2^{2n-2k-1}\frac{k!}{\left(2n+1\right)!}\sum_{j=0}^{n-1}B_{2n-2j}\binom{2k-2j-1}{k-2j-1}\binom{2n+1}{2j+1}  .
\end{align*}
The right-hand side of \eqref{eq:first_lommel} simplifies to
\[
\left(-1\right)^{n-1}\binom{k-1-n}{n-1}\frac{\Gamma\left(k+\frac{1}{2}-n+1\right)}{\Gamma\left(\frac{1}{2}+1+n\right)}\frac{1}{2^{2n}}=\left(-1\right)^{n-1}2^{2n-2k}\binom{k-1-n}{n-1}\frac{\left(2k-2n+1\right)!n!}{\left(k-n\right)!\left(2n+1\right)!}.
\]
We deduce
\begin{align*}
\sum_{j=0}^{n-1}B_{2n-2j}\binom{2k-2j-1}{k}\binom{2n+1}{2j+1} & =\frac{1}{2}\binom{k-1-n}{n-1}\frac{\left(2k-2n+1\right)!n!}{\left(k-n\right)!k!}\\
 & =\frac{1}{2}\frac{n\left(2k-2n+1\right)}{k-n}\binom{2k-2n}{k}.
\end{align*}
Adding an extra $j=n$ term in the sum yields
\begin{align*}
\sum_{j=0}^{n}B_{2n-2j}\binom{2k-2j-1}{k}\binom{2n+1}{2j+1}&=\frac{1}{2}\frac{n\left(2k-2n+1\right)}{k-n}\binom{2k-2n}{k}+\binom{2k-2n-1}{k-2n-1} \\
&=\frac{1}{2}\left(2n+1\right)\binom{2k-2n}{k}.
\end{align*}
This coincides with the Gessel-Viennot identity, since for $k>n,$
\[
\sum_{j=0}^{\left\lfloor \frac{k-1}{2}\right\rfloor }B_{2n-2j}\binom{2k-2j-1}{k}\binom{2n+1}{2j+1}=\sum_{j=0}^{n}B_{2n-2j}\binom{2k-2j-1}{k-2j-1}\binom{2n+1}{2j+1}.
\]
\end{remark}

\section{Airy Zeta}\label{section5}
\subsection{Computation of elementary values}
The Airy function has Weierstrass factorization \cite{Flajolet}
\begin{equation}
\Ai\left(z\right)=\Ai\left(0\right)e^{\frac{\Ai'\left(0\right)}{\Ai\left(0\right)}z}\prod_{n\ge1}\left(1-\frac{z}{a_{n}}\right)e^{\frac{z}{a_{n}}}\label{eq:Weierstrass},
\end{equation}
where all the zeros $\left\{ a_{n}\right\} $ are real and negative. We frequently use the constants
$$\Ai\left(0\right)=\frac{1}{3^{\frac{2}{3}}\Gamma\left(\frac{2}{3}\right)},\,\ \Ai'\left(0\right)=-\frac{1}{3^{\frac{1}{3}}\Gamma\left(\frac{1}{3}\right)}.$$
We can now define the standard associated zeta and multiple zeta functions, hereforth denoted by $\zeta_{\Ai}$, and calculate MZVs using standard techniques. Using our product/sum duality, the values $\zeta_{\Ai}(\{2\}^n)$ are obtained from a series expansion for $\Ai(z)\Ai(-z)$ while the values $\zeta_{\Ai}(\{4\}^n)$ are deduced from a series expansion for $\Ai(z)\Ai(-z)\Ai(i z)\Ai(-i z)$.

\begin{thm}\cite{Zhang}
The Airy MZV is equal to
\begin{align}
\zeta_{\Ai}\left(\text{\ensuremath{\left\{  2\right\}}}^{n}\right) & =\frac{1}{12^{\frac{n}{3}}n!\left(\frac{5}{6}\right)_{\frac{n}{3}}}.\label{eq:Ai(2^n)}
\end{align}
\end{thm}

The usual proof (\cite{Zhang})
uses the series expansion \cite[(3.9)]{Reid}
\begin{equation}
\Ai\left(x\right)\Ai\left(-x\right)=\frac{2}{\sqrt{\pi}}\sum_{n\ge0}\frac{\left(-1\right)^{n}x^{2n}}{12^{\frac{2n+5}{6}}n!\Gamma\left(\frac{2n+5}{6}\right)},\label{eq:series}
\end{equation}
deduced by Reid as a consequence of the integral representation
\[
\Ai\left(x\right)\Ai\left(-x\right)=\frac{1}{\pi2^{\frac{1}{3}}}\int_{-\infty}^{+\infty}\Ai\left(2^{-\frac{4}{3}}t^{2}\right)\cos\left(xt\right)dt.
\]
Then the Weierstrass factorization \eqref{eq:Weierstrass} allows us to deduce the generating function of $\zeta_{\Ai}\left(\text{\ensuremath{\left\{  2\right\} } }^{n}\right)$ as
\begin{align*}
\sum_{n=0}^\infty \zeta_{\Ai}\left(\text{\ensuremath{\left\{  2\right\}}}^{n}\right)z^{n} & =\prod_{n\ge1}\left(1+\frac{z}{a_{n}^{2}}\right)=\frac{\Ai\left(i\sqrt{z}\right)\Ai\left(-i\sqrt{z}\right)}{\Ai\left(0\right)^{2}}.
\end{align*}
Finally, notice that $\zeta_{\Ai}\left(\left\{ 2\right\}^{0}\right)=1$ to obtain the desired result.


We propose here another approach to this result that does not require Reid's expansion
and reveals a surprising link with results from the previous section on Bessel MZVs. Looking at $\zeta_{\Ai}\left(\left\{2\right\}^{3n}\right)$ in \eqref{eq:Ai(2^n)},
we recognize, up to a factor $\left(\frac{3}{2}\right)^{4n},$ the
value of $\zeta_{B,\nu=-\frac{1}{3}}\left(\left\{  4\right\}^{n}\right)$
computed previously for the Bessel zeta function; this is not a coincidence.
\begin{thm}
The Airy MZV and the Bessel MZV are related by
\begin{equation}
\zeta_{\Ai}\left(\left\{ 2\right\} ^{3n}\right)=\left(\frac{2}{3}\right)^{4n}\zeta_{-\frac{1}{3}}\left(\left\{ 4\right\} ^{n}\right)
\end{equation}
and
\begin{equation}
\frac{\zeta_{\Ai}\left(\left\{ 2\right\} ^{3n+1}\right)}{\zeta_{\Ai}\left(\left\{ 2\right\} \right)}=\left(\frac{2}{3}\right)^{4n}\zeta_{\frac{1}{3}}\left(\left\{ 4\right\} ^{n}\right).
\end{equation}
\end{thm}

\begin{proof}
We begin by studying the decomposition
\[
\Ai\left(z\right)=\frac{\sqrt{z}}{3}\left(I_{-\frac{1}{3}}\left(\xi\right)-I_{\frac{1}{3}}\left(\xi\right)\right),
\]
with $\xi=\frac{2}{3}z^{\frac{3}{2}}$. The Airy function is an entire
function 
\[
\Ai\left(z\right)=\sum_{n\ge0}a_{n}z^{n}
\]
such that every third term vanishes, i.e. $a_{3n+2}=0.$ The coefficients $a_{3n}$ are provided
by the term $\frac{\sqrt{z}}{3}I_{-\frac{1}{3}}\left(\xi\right)$,
while the coefficients $a_{3n+1}$ arise from $\frac{\sqrt{z}}{3}I_{\frac{1}{3}}\left(\xi\right)$. Similarly, the two terms in the
expansion
\[
\Ai\left(-z\right)=\frac{\sqrt{z}}{3}\left(J_{-\frac{1}{3}}\left(\xi\right)-J_{\frac{1}{3}}\left(\xi\right)\right)
\]
provide the coefficients $a_{3n}$ and $a_{3n+1}$ respectively. When multiplying $\Ai\left(z\right)$
by $\Ai\left(-z\right),$ the $z^{3n}$ terms can only arise through a contribution
from the first term in $\Ai\left(z\right)$ and the first in $\Ai\left(-z\right)$, and so on. 
In terms of the normalized Bessel functions
\[
j_{\nu}\left(z\right)=2^{\nu}\Gamma\left(\nu+1\right)\frac{J_{\nu}\left(z\right)}{z^{\nu}},\thinspace\thinspace i_{\nu}\left(z\right)=2^{\nu}\Gamma\left(\nu+1\right)\frac{I_{\nu}\left(z\right)}{z^{\nu}},
\]
remarking that
$i_{\nu}\left(z\right)=j_{\nu}\left(i z\right),$
we deduce \footnote{This simplification does not happen in the case of terms congruent to $2 \pmod{3}$ and we obtain the surprising expansion
\[
\sum_{n\ge0}\zeta_{\Ai}\left(\text{\ensuremath{\left\{  2\right\} } }^{3n+2}\right)z^{6n}=\frac{\Gamma^{2}\left(\frac{2}{3}\right)}{\left(4\pi\right)3^{\frac{1}{6}}}\thinspace\thinspace_{0}F_{3}\left(\begin{array}{c}
-\\
\frac{8}{6},\frac{9}{6},\frac{10}{6}
\end{array};\frac{z^{6}}{324}\right).
\]
}

\begin{align*}
\sum_{n\ge0}\left(-1\right)^{n}\zeta_{\Ai}\left(\text{\ensuremath{\left\{  2\right\} } }^{3n}\right)\left(\frac{3z}{2}\right)^{4n} & =j_{-\frac{1}{3}}\left(i z\right)j_{-\frac{1}{3}}\left(z\right)
\end{align*}
and
\begin{align*}
\sum_{n\ge0}\left(-1\right)^{n}\frac{\zeta_{\Ai}\left(\text{\ensuremath{\left\{  2\right\} } }^{3n+1}\right)}{\zeta_{\Ai}\left(\text{\ensuremath{\left\{  2\right\} } }\right)}\left(\frac{3z}{2}\right)^{4n} & =j_{\frac{1}{3}}\left(i z\right)j_{\frac{1}{3}}\left(z\right).
\end{align*}

Both products on the right hand sides have a Weierstrass factorization, namely
\[
j_{-\frac{1}{3}}\left(i z\right)j_{-\frac{1}{3}}\left(z\right)=\prod_{k\ge1}\left(1-\frac{z^{4}}{z_{k,-\frac{1}{3}}^{4}}\right)
\]
and
\[
j_{\frac{1}{3}}\left(i z\right)j_{\frac{1}{3}}\left(z\right)=\prod_{k\ge1}\left(1-\frac{z^{4}}{z_{k,\frac{1}{3}}^{4}}\right).
\]
These factorizations are generating functions for Bessel multiple
zeta values, so that we deduce the result.
\end{proof}

More generally, we can apply this multisection technique whenever the function we are considering can be expressed as a linear combination of other simple functions symmetrical enough that they have zero series coefficients for certain congruence classes. 
For instance, this may allow us to study the Bessel $Y$ or $K$ functions, or the Hankel functions, which are all linear combinations of other Bessel functions. 

These techniques both carry over into the study of the derivative of the Airy function.
\begin{thm}
We have 
\[
\zeta_{\Ai'}\left(\left\{ 2\right\} ^{n}\right)=\frac{1}{n!}\frac{1}{2^{\frac{2n+1}{3}}3^{\frac{2n-3}{6}}}\frac{\Gamma\left(\frac{1}{3}\right)^{2}}{\Gamma\left(\frac{1}{2}\right)\Gamma\left(\frac{2n+1}{6}\right)}.
\]
\end{thm}
\begin{proof}
The product $\Ai'\left(z\right)\Ai'\left(-z\right)$ is equal to
\[
\Ai'\left(z\right)\Ai'\left(-z\right)=-\frac{1}{2}\frac{d^{2}}{dz^{2}}\left[\Ai\left(z\right)\Ai\left(-z\right)\right],
\]
since we can calculate
\[
\frac{d^{2}}{dz^{2}}\left[\Ai\left(z\right)\Ai\left(-z\right)\right]=\Ai''\left(z\right)\Ai\left(-z\right)-2\Ai'\left(z\right)\Ai'\left(-z\right)+\Ai\left(z\right)\Ai''\left(-z\right),
\]
and use the fact that $\Ai''\left(z\right)=z\Ai\left(z\right)$. This can then be used in conjunction with the series expansion \eqref{eq:series} to derive the given result.
\end{proof}
As with the previous multisection technique, we could have started with the expansion
\[
\Ai'\left(z\right)=\frac{z}{3}\left(I_{\frac{2}{3}}\left(\frac{2}{3}z^{\frac{3}{2}}\right)-I_{-\frac{2}{3}}\left(\frac{2}{3}z^{\frac{3}{2}}\right)\right)
\]
and studied different dissections modulo $3$ to obtain the same result.

A closed form for the higher-order multiple values $\zeta_{\Ai}\left(\left\{ 2p \right\} ^{n}\right)$ with $p>2$ seems out of reach; however, we were able to obtain the following result.

\begin{thm}
\label{thm:zeta4n}
The values of $\zeta_{\Ai}\left(\left\{ 4 \right\} ^{n}\right)$ can be computed as the convolution 
\begin{equation}
\label{zetaAi4n}
\zeta_{\Ai}\left(\left\{ 4 \right\}^{n}\right) = 
\frac{\left( -1 \right)^{n}}{12^{\frac{2n}{3}}2n!}
\sum _{k=0}^{2 n} 
\binom{2n}{k}
\frac{(-1)^k}{\left(\frac{5}{6}\right)_{\frac{k}{3}}  
\left(\frac{5}{6}\right)_{\frac{2n-k}{3}}}
\end{equation}
and are equal to
\begin{align*}
\zeta_{\Ai}\left(\left\{ 4 \right\}^{n}\right)&=
\frac{\left( -1 \right)^n}{12^{\frac{2n}{3}}\left( 2n \right)!}
\Gamma\left(\frac{5}{6}\right)
\frac{\pFq{4}{3}
{\frac{1}{6}-\frac{2n}{3},
\frac{1}{3}-\frac{2n}{3},
-\frac{2n}{3},
\frac{2}{3}-\frac{2n}{3}}
{\frac{1}{3},\frac{2}{3},\frac{5}{6}}{-1}}{\Gamma\left(\frac{2n}{3}+\frac{5}{6}\right)}\\
&
-\frac{6\left( -1 \right)^n}{12^{\frac{2n}{3}}\left( 2n-1 \right)!}\Gamma^{2}\left(\frac{5}{6}\right)
\frac{\pFq{4}{3}
{\frac{1}{3}-\frac{2n}{3},\frac{1}{2}-\frac{2n}{3},\frac{2}{3}-\frac{2n}{3},1-\frac{2n}{3}}
{\frac{2}{3},\frac{7}{6},\frac{4}{3}}{-1}}
{\Gamma\left(\frac{1}{6}\right)\Gamma\left(\frac{2n}{3}+\frac{1}{2}\right)}\\
&
+\frac{\left( -1 \right)^n}{12^{\frac{2n}{3}}\left( 2n-2 \right)!}\Gamma^{2}\left(\frac{5}{6}\right)
\frac{\pFq{4}{3}
{\frac{2}{3}-\frac{2n}{3},\frac{5}{6}-\frac{2n}{3},\frac{4}{3}-\frac{2n}{3},1-\frac{2n}{3}}
{\frac{4}{3},\frac{3}{2},\frac{5}{3}}
{-1}}
{\Gamma\left( \frac{1}{2} \right)\Gamma\left(\frac{2n}{3}+\frac{1}{6}\right)}.
\end{align*}
\end{thm}

\begin{proof}
We give only an outline of the proof here: divide the sum \eqref{zetaAi4n} into three parts according to the residue class of $k$ modulo 3: this gives three terms which can be checked to coincide with those of the result.
For example, keeping only terms divisible by $3$ in \eqref{zetaAi4n} gives
\[
\frac{\left( -1 \right)^{n}}{12^{\frac{2n}{3}}2n!}
\sum _{k=0}^{2 n} 
\binom{2n}{3k}
\frac{(-1)^{3k}}{\left(\frac{5}{6}\right)_{\frac{3k}{3}}  
\left(\frac{5}{6}\right)_{\frac{2n-3k}{3}}}
= 
\frac{\left( -1 \right)^n}{12^{\frac{2n}{3}}\left( 2n \right)!}
\Gamma\left(\frac{5}{6}\right)
\frac{\pFq{4}{3}
{\frac{1}{6}-\frac{2n}{3},
\frac{1}{3}-\frac{2n}{3},
-\frac{2n}{3},
\frac{2}{3}-\frac{2n}{3}}
{\frac{1}{3},\frac{2}{3},\frac{5}{6}}{-1}}{\Gamma\left(\frac{2n}{3}+\frac{5}{6}\right)},
\]
as can be checked by standard summation techniques.
\end{proof}

\subsection{Airy Bernoulli numbers}
Similarly to the case of the Bessel zeta function, we introduce Airy Bernoulli numbers in order to compute the Airy multiple zeta star values. Although these numbers do not have a closed form expression, they satisfy a simple recurrence identity that allows their effective numerical computation. 

Define the Airy Bernoulli numbers $\mathcal{B}_n$ by the generating function
\begin{equation}
\sum_{n\ge0}\frac{\mathcal{B}_n}{n!}z^{n}=\frac{\Ai\left(0\right)}{\Ai\left(z\right)}.\label{eq:Bernoulli generating}
\end{equation}
For example,
\[
\mathcal{B}_{0}=1,\thinspace\thinspace \mathcal{B}_{1}=3^{\frac{1}{3}}\frac{\Gamma\left(\frac{2}{3}\right)}{\Gamma\left(\frac{1}{3}\right)}=-\frac{\Ai'\left(0\right)}{\Ai\left(0\right)},\thinspace\thinspace\frac{\mathcal{B}_{2}}{2!}=\mathcal{B}_{1}^{2}=3^{\frac{2}{3}}\frac{\Gamma^{2}\left(\frac{2}{3}\right)}{\Gamma^{2}\left(\frac{1}{3}\right)},
\]
\[
\frac{\mathcal{B}_{3}}{3!}=-\frac{1}{6}+3\frac{\Gamma^{3}\left(\frac{2}{3}\right)}{\Gamma^{3}\left(\frac{1}{3}\right)},\thinspace\thinspace\frac{\mathcal{B}_{4}}{4!}=-\frac{3^{\frac{1}{3}}}{4}\frac{\Gamma\left(\frac{2}{3}\right)\left[\Gamma^{3}\left(\frac{2}{3}\right)-12\Gamma^{3}\left(\frac{1}{3}\right)\right]}{\Gamma^{4}\left(\frac{1}{3}\right)}
\]
and so on. This parallels the introduction of Bessel-Bernoulli numbers as coefficients in the series expansion of the reciprocal Bessel 
function \cite{Frappier}. We have the following recurrences to compute these numbers and then the Airy zeta function:
\begin{thm}
Define the sequence
\[
a_{3n}=\frac{\left(-1\right)^{n}\left( 3n \right)!}{3^{2n}n!\left(\frac{2}{3}\right)_{n}},
\thinspace\thinspace a_{3n+1}=\frac{\left(-1\right)^{n+1}\Gamma\left(\frac{2}{3}\right)\left(3n+1\right)!}{3^{2n+\frac{2}{3}}n!\Gamma\left(n+\frac{4}{3}\right)},
\thinspace\thinspace a_{3n+2}=0.
\]
The Airy Bernoulli numbers satisfy the linear recurrence
\begin{equation}
\sum_{k=0}^{n}\binom{n}{k}\mathcal{B}_{k}a_{n-k}=n!\delta_{n},
\nonumber
\end{equation}
so that 
\begin{equation}
\mathcal{B}_{n}=\begin{cases}
-\sum_{k=0}^{n-1}\binom{n}{k}\mathcal{B}_{k}a_{n-k}, & n>0\\
1 & n=0
\end{cases}.
\label{eq:recurrence1}
\end{equation}
We then have the linear recurrence between the Airy Bernoulli numbers and the Airy zeta function
\begin{equation}
\frac{\mathcal{B}_{n+1}}{n!}=-\frac{\Ai'\left(0\right)}{\Ai\left(0\right)}\frac{\mathcal{B}_{n}}{n!}+\sum_{r=0}^{n-1}\frac{\mathcal{B}_{r}}{r!}\zeta_{\Ai}\left(n+1-r\right),\thinspace\thinspace n\ge1,
\nonumber
\end{equation}
so that
\begin{equation}
\zeta_{\Ai}\left(n+1\right)=\frac{\Ai'\left(0\right)}{\Ai\left(0\right)}\frac{\mathcal{B}_{n}}{n!}+\frac{\mathcal{B}_{n+1}}{n!}-\sum_{r=1}^{n-1}\frac{\mathcal{B}_{r}}{r!}\zeta_{\Ai}\left(n+1-r\right),\thinspace\thinspace n\ge1.
\label{eq:recurrence2}
\end{equation}
\end{thm}

\begin{proof}
The first recurrence \eqref{eq:recurrence1} is a consequence of the
identity
\[
\frac{\Ai\left(0\right)}{\Ai\left(z\right)}\frac{\Ai\left(z\right)}{\Ai\left(0\right)}=1.
\]
The second identity is obtained by taking the derivative of \eqref{eq:Bernoulli generating}
to obtain
\begin{equation}
\label{product generating}
-\Ai\left(0\right)\frac{\Ai'\left(z\right)}{\Ai^{2}\left(z\right)}=\sum_{n\ge1}\frac{\mathcal{B}_{n}}{\left(n-1\right)!}z^{n-1},
\end{equation}
and writing
\[
-\Ai\left(0\right)\frac{\Ai'\left(z\right)}{\Ai^{2}\left(z\right)}=-\frac{\Ai\left(0\right)}{\Ai\left(z\right)}\frac{\Ai'\left(z\right)}{\Ai\left(z\right)}.
\]
Recognizing the product of the generating functions of the Airy Bernoulli numbers
and the Airy zeta function in the right-hand side and comparing coefficients of $z^{n}$ in \eqref{product generating} gives the result.
\end{proof}
\begin{rem}
As in the hypergeometric and Bessel cases, the linear recurrences \eqref{eq:recurrence1} and \eqref{eq:recurrence2} allow us to recursively compute the Bernoulli
numbers $\mathcal{B}_{1},\mathcal{B}_{2},\dots$  and the values of the Airy zeta function respectively. And similarly to the hypergeometric and Bessel cases, we also deduce
\[
\zeta^{*}_{\Ai}\left( \left\{2\right\}^n \right) = \sum_{k=0}^{n} \binom{n}{k}
\left( -1 \right)^k \mathcal{B}_{k}\mathcal{B}_{n-k}.
\]
\end{rem}


\section{Conclusion}\label{section5}
These three extensions of the usual MZVs shed new light on the two questions expressed in the introduction. First, we saw that after fixing a set of zeros, the details of the Weierstrass factorization of the associated entire function are essential to correctly normalizing the resulting MZVs. Second, the systematic introduction of Bernoulli numbers as coefficients of the reciprocal of the MZV generating function gives a practical, two-step method for the evaluation of MZSVs: first derive the linear recurrence satisfied by these Bernoulli numbers, then deduce the star MZVs using another set of linear recurrences.
It may also happen that, besides satisfying linear recurrences, these Bernoulli numbers have nice integral representations (see \cite{Christophe2} in the case of hypergeometric Bernoulli numbers) that provide additional information on their properties. An obvious follow-up to this study is the extension of the approaches developed here to special functions having higher-order Weierstrass factorizations. 

\section*{Acknowledgements}
None of the authors have any competing interests in the manuscript.
This material is based upon work supported by the National Science Foundation under Grant No. DMS-1439786 while the second author was in residence at the Institute for Computational and Experimental Research in Mathematics in Providence, RI, during the Point Configurations in Geometry, Physics and Computer Science Semester Program, Spring 2018. We thank the anonymous referee for his excellent and straightforward suggestions.










\bibliographystyle{plain}

\end{document}